\documentclass{amsproc}      
\usepackage{amsfonts,amssymb,mathrsfs}         
\usepackage{graphicx}        
\usepackage{amsmath}         
\usepackage{epsfig}
\usepackage{color}
\usepackage{tikz,caption,subcaption}
\usepackage{mathtools}
\usepackage{enumerate}
\newtheorem{theorem}{Theorem}[section]
\newtheorem{lemma}[theorem]{Lemma}

\theoremstyle{definition}
\newtheorem{definition}[theorem]{Definition}

\theoremstyle{remark}

\newcommand{\N}{\mathbb{N}}              
\newcommand{\R}{\mathbb{R}}              
\newcommand{\hd}{\mathcal{H}}            
\newcommand{\B}{\mathcal{B}}             
\newcommand{\de}{\delta}                 
\newcommand{\De}{\Delta}                 
\newcommand{\ldq}{\textquotedblleft}     
\newcommand{\E}{\mathcal{E}}             
\newcommand{\f}{\mathbf{f}}              
\newcommand{\M}{\mathcal{M}}             %

\newcommand{\s}{\mathcal{S}}             
\newcommand{\dm}{\mathrm{d}\mu}
\newcommand{\la}{\lambda}                

\newcommand{\dom}{\text{dom}}            
\newcommand{\e}{\epsilon}                

\newcommand{\pa}{\partial}
\newcommand{\al}{\alpha}

\newcommand{\uline}{\underline}
\newcommand{\oline}{\overline}

\newcommand*\rfrac[2]{{}^{#1}\!/_{#2}} 

\numberwithin{equation}{section}

\begin{document}

\title[THE p-LAPLACIAN ON THE SIERPI\'NSKI GASKET]{EXISTENCE OF MULTIPLE SOLUTIONS OF A p-LAPLACIAN EQUATION ON THE SIERPI\'NSKI GASKET}

\author[A. Sahu]{Abhilash Sahu}
\address{Department of Mathematics, Indian Institute of Technology Delhi, Hauz Khas, New Delhi 110016, India}
\email{sahu.abhilash16@gmail.com}
\author[A. Priyadarshi]{Amit Priyadarshi}
\address{Department of Mathematics, Indian Institute of Technology Delhi, Hauz Khas, New Delhi 110016, India}
\email{priyadarshi@maths.iitd.ac.in}


\date{}

\begin{abstract}
In this paper we study the following boundary value problem involving the weak p-Laplacian.
\begin{equation*}
 \quad -M(\|u\|_{\mathcal{E}_p}^p)\Delta_p u = h(x,u) \; \text{in}\; \mathcal{S}\setminus\mathcal{S}_0; \quad u = 0 \; \mbox{on}\; \mathcal{S}_0,
\end{equation*}
where $\mathcal{S}$ is the Sierpi\'nski gasket in $\R^2$, $\mathcal{S}_0$ is its boundary. $M : \R \to \R$ defined by $M(t) = at^k +b$ and $a,b,k >0$ and $h : \s \times \R \to \R.$ We will show the existence of two nontrivial weak solutions to the above problem.
\end{abstract}
\maketitle

\section{Introduction}
In this article we will discuss the existence of weak solutions of the boundary value problem on the Sierpi\'nski gasket.
\begin{equation}\label{prob}
\begin{split}
  \quad -M(\|u\|_{\mathcal{E}_p}^p)\Delta_p u &= h(x,u) \; \text{in}\; \mathcal{S}\setminus\mathcal{S}_0, \\
  u &= 0  ~~ \text{on}~ \s_0,
 \end{split}
\end{equation}
where $\s$ is the Sierpi\'nski gasket in $\R^2$ and $\s_0$ is the boundary of the Sierpi\'nski gasket. $M : \R \to \R$ is defined by $M(t) = at^k +b$ for $a,b,k >0$ and $h : \s \times \R \to \R$ is defined as $h(x,u) = \la f(x)|u|^{q-2}u + g(x)|u|^{l-2}u,$ where $1<q<p<l$ and $\la >0.$
$\De_p$ denotes the weak $p$-Laplacian where $p>1.$ We will discuss more about this in the next section. If $u \in \dom_0(\E_p)$ (defined in the next section) and satisfies
 $$\lambda \int\limits_{\s} f(x)|u|^{q-2}uv \dm + \int\limits_{\s} g(x)|u|^{l-2}u v \dm \in M(\|u\|_{\E_p}^p)\E_p(u,v)$$
 for all $v \in \dom_0(\E_p),$ then we will call $u$ to be a weak solution of \eqref{prob}.

In this section we will discuss about brief literature review of Laplacian and p-Laplacian on fractal domains in particular on the Sierpi\'nski gasket and then a brief literature of Kirchhoff type equations on regular domains, that is, open connected domains with smooth boundary. Then we will give the arrangement of the paper.

The name \ldq fractal" was introduced by B. B. Mandelbrot for the first time in 1975.
Mandelbrot argued that fractal models occur in physics, biology, mechanics etc. The type of equations defined above occur in problems like elastic fractal media, fluid flow through fractal regions, reaction diffusion equations and waves on fractal objects. For more details about fractals and analysis on fractals readers are encouraged to read \cite{KF,kiga1,RS} and references theirin. These are some of the reasons due to which study of differential equations on fractal domains became fascinating for the past few decades.

An extensive amount of research has been done and still continuing on Laplacian and p-Laplacian on regular domains but compared to this a very less amount of literature is available on fractal domains. Probably the reason for this is that there is no well-known notion of Laplacian and p-Laplacian on general fractal domains. We can define the Laplacian on the Sierpi\'nski gasket and some p.c.f. fractals as in \cite{kiga1, kiga2, str} and we can define the p-Laplacian on the Sierpi\'nski gasket as in \cite{sw, hps} (which we have detailed in the next section).

Kigami and Lapidus \cite{KL} studied the eigenvalue problem on the Sierpi\'nski gakset and established an analogue of Weyl's classical theorem for the asymptotics of eigenvalues of Laplacians on p.c.f. self-similar fractals. In \cite{AT} Teplyaev studied the spectral properties of the Laplacian on infinite Sierpi\'nski gasket. He has shown that Laplacian with the Neumann boundary condition has pure point
spectrum and the set of eigenfunctions with compact support is complete. Falconer \cite{falc} studied the semilinear equations $\De u + u^p =0$ with zero Dirichlet boundary conditions and showed that they have nontrivial nonnegative solutions if $0 <\nu \leq 2$ and $p>1$ or if $\nu > 2$ and $1 < p < (\nu+2)/(\nu-2),$ where $\nu$ is the spectral dimension of the system. In \cite{fh} Falconer and Hu studied nonlinear diffusion equation $\pa u/\pa t = \De u + u^p$ where $p > 1$ on certain unbounded fractal domains. They have shown that there are nonnegative global solutions for nonnegative initial data if $p > 1+2/d_s,$ while solutions
blow up if $p \leq 1 + 2/d_s,$ where $d_s$ is the spectral dimension of the domain. In \cite{BRV} the authors have studied the problem 
$$ -\De u(x) = \la f(x,u(x)) + \eta g(x,u(x))\hspace{.21in} \forall x \in V\setminus V_0; ~~~~u|_{V_0} = 0$$ using Ricceri type three critical point theorem and have shown the existence of at least three solutions of this problem. In \cite{BRV2} the authors have shown the existence of infinitely many weak solutions to the elliptic problem $\De u(x) + a(x)u(x) = g(x)f(u(x))$ on the Sierpi\'nski gasket with zero Dirichlet boundary condition. Stancu-Dumitru in \cite{denisa1} studied the problem
\begin{align*}
-\De u(x) &= f(x)|u(x)|^{p-2}u(x) +(1-g(x))|u(x)|^{q-2}u(x) &\text{~for~~} x \in  V\setminus V_0\\
u &= 0 &\text{~for~~} x \in V_0
\end{align*}
where $\De$ is the Laplacian on $V, 1 < p < 2 < q$ are real numbers, $f, g \in C(V)$ satisfy 
$f^+ = \max\{f,0\} \neq 0$ and $ 0\leq g(x) < 1$ for all $x \in V.$ In \cite{denisa2} she studied 
\begin{align*}
-\De u(x) + g(u(x))& = \la f(u(x))&\text{~for~~} x \in  V\setminus V_0\\
u &= 0 &\text{~for~~} x \in V_0
\end{align*}
where $\la$ is a positive real number and $f, g$ are given functions. In both the problems she used variational method and Nehari manifold technique to show the existence of solutions. Molica Bisci and R\u adulescu \cite{br} proved a characterization result on the existence of nonnegative and nonzero strong solutions for the Dirichlet problem
$$ \De u(x) = \la \al(x)f(u(x)) \text{~~for~~} x \in V\setminus V_0; u(x) = 0 \text{~~for~~} x \in V_0$$
where V stands for the Sierpi\'nski gasket in $(\R^{N-1}, |\cdot|), ~N \geq 2, ~V_0$ is its intrinsic
boundary (consisting of its $N$ corners) and $\De$ denotes the weak Laplacian. Under suitable conditions on $f$ they have shown the existence of solutions. In \cite{BC} Breckner and Chill have studied a problem with Laplace operator on the Sierpi\'nski gasket with nonlinear Robin
boundary conditions and have shown that for certain Robin boundary conditions the
Laplace operator generates a positive, order preserving, L${^\infty}$-contractive semigroup
which is sandwiched (in the sense of domination) between the semigroups generated
by the Dirichlet-Laplace operator and the Neumann-Laplace operator. For p-Laplacian operator Strichartz and Wong \cite{sw} analyzed the problem $\De_p u = f$ with prescribed boundary values by solving an equivalent minimization problem and also gave numerical solution to this problem.
Priyadarshi and Sahu in \cite{PS} studied the problem
\begin{equation*}
 \quad -\De_p u = \la a(x)|u|^{q-1} u + b(x)|u|^{l-1}u \; \text{in}\; \s\setminus\s_0; \quad u = 0 \; \mbox{on}\; \s_0,
\end{equation*}
and in \cite{SP} studied the problem 
\begin{equation*}
 \quad -\Delta_p u = \lambda a(x)|u|^{p-2} u + b(x)|u|^{\ell-1}u \; \text{in}\; \mathcal{S}\setminus\mathcal{S}_0; \quad u = 0 \; \mbox{on}\; \mathcal{S}_0,
\end{equation*}
where $\s $ is the Sierpi\'nski gasket in $\R^2$, $\s_0$ is its boundary and $\la >0.$ Under suitable conditions, they have shown the existence of at least two nontrivial weak solutions and one weak solution, respectively to the above problems for a certain range of $\la.$ Many authors have studied different types of equations involving Laplacian on the Sierpi\'nski gasket, we cite a few of them which are related to this article \cite{BMR,falc1,FMR,BRS}.

Now we will discuss about Kirchhoff type equations on regular domains. Consider the following problem 
\begin{equation}\label{0.1}
-M\left(\int_{\Omega} |\nabla u|^2 {dx}\right) \De u =\la f(x,u) \text{~in~} \Omega ; u(x) = 0 \text{~on~} \pa\Omega 
\end{equation}
where $\Omega$ is a smooth bounded domain in $\R^N$ and $\pa\Omega$ is its boundary.
In \cite{ACM} the authors have studied Eq. \eqref{0.1} where they have assumed $M$ is a positive function, $f$ has subcritical growth and $\la=1.$ They have shown the existence of positive solutions to this above class of functions.  Perera and Zhang studied the problem \eqref{0.1} for $M(t) = at +b$ where $a,b>0$ and $\la = 1.$ In \cite{PZ} they have shown the existence of solutions using the Yang index and critical groups with the assumptions that $f$ is a Carath\'eodory function satisfying some growth conditions in second variable. In \cite{ZP} they have studied the same problem using variational methods and invariant sets of descent flow and have obtained a sign-changing solutions assuming the hypothesis that $f(x,t)$ is locally Lipschitz continuous in $t \in \R,$ uniformly in $x \in \overline{\Omega}$ and subcritical. He and Zou \cite{HZ} obtained infinitely many almost everywhere positive weak solutions to a class of Kirchhoff-type problem \eqref{0.1} under the assumptions that $M(t)= at+b$ and $f(x,t)$ is a Carath\'eodory function satisfying some conditions. In \cite{LZW} Liao et al. studied the following problem 
\begin{align*}
-\left(a+ b \int_{\Omega}|\nabla u|^2dx\right) \De u  &= \nu u^3 + Q(x)u^q \text{~~in~~} \Omega \\
u &= 0 \text{~~on~~} \pa\Omega 
\end{align*}
where $\Omega \subset \R^3$ is a bounded domain, $a,b \geq 0$ and $a + b > 0, \nu > 0, 3<q\leq 5$ and $Q(x)>0$ are four parameters. They have used mountain pass lemma to show the existence of positive
solutions. Many other researchers have studied Kirchhoff type equations in \cite{MZ,ST,cheng,CKW,LLS,fig}. To the best of our knowledge there is no literature for Kirchhoff type equations on the Sierpi\'nski gasket. This motivated us to study the Kirchhoff type equations on the Sierpi\'nski gasket.
 
The outline of our paper is as follows. In section 2 we discuss about the weak $p$-Laplacian on the Sierpi\'nski gasket and also describe how we are going from the energy functional $\E_p(u)$ to the energy form $\E_p(u,v)$. We recall some important results and state our main theorem. In section 3 we define the Euler functional associated to our problem \eqref{prob} and study some of its properties. In section 4 we do the analysis of the fibering map $\phi_u$ and also find the range of $\la>0$ for which problem \eqref{prob} has two nontrivial solutions. Finally, in section 5 we give the detailed proof of our theorem stated in section 2.
\section{Preliminaries and Main results}
We will start by introducing the Sierpi\'nski gasket. Let $\s_0 = \{q_1, q_2, q_3\}$ be three points on $\R^2$ equidistant from each other. Let $F_i(x) = \frac{1}{2}(x-q_i) + q_i$ for $i= 1,2,3$ and $F(A) = \cup_{i=1}^3 F_i(A)$ for $A \subseteq \R^2$.  It is well known that $F$ has a unique fixed point $\s$ (see, for instance, \cite[Theorem 9.1]{KF}), which is called the Sierpi\'nski gasket. Another way to view the same is $\s = \oline{\cup_{j \geq 0}F^{j}(\s_0)},$ where $F^j$ means $F$ composed with itself $j$ times. We know that $\s$ is a compact set in $\R^2$ and we will use certain properties of functions on $\s$ due to the compactness of the domain. It is well known that the Hausdorff dimension of $\s$ is $\frac{\ln 3}{\ln 2}$ and the $\frac{\ln 3}{\ln 2}$-dimensional Hausdorff measure is finite and nonzero (i.e., $0<\hd^{\frac{\ln 3}{\ln 2}}(\s)<\infty)$ (see, \cite[Theorem 9.3]{KF}). Throughout this paper, we will use this measure and denote it by $\mu$. If $f$ is a measurable function on $\s$, then
$$\|f\|_\infty \coloneqq \inf \{a \in \R : \mu\{x \in \s : |f(x)| > a\} = 0\}.$$
We define the $p$-energy  with the help of a three variable function $A_p$ which is convex, homogeneous of degree $p$ and is invariant under addition of constant and permutation of indices. The $m^{\text{th}}$ level Sierpi\'nski gasket is $\s^{(m)} = \cup_{j=0}^m F^j(\s_0)$. We construct the $m^{\text{th}}$ level crude energy as $$E_p^{(m)}(u) = \sum_{|\omega| = m} A_p\left(u(F_\omega q_1), u(F_\omega q_2), u(F_\omega q_3)\right)$$ and the $m^{\text{th}}$ level renormalized $p$-energy is given by $$\E_p^{(m)}(u) = (r_p)^{-m} E_p^{(m)}(u),$$ where $r_p$ is the unique (with respect to p), independent of $A_p$, renormalizing factor and $0 < r_p <1$. For more detail see \cite{hps}. Now we can observe that $\E_p^{(m)}(u)$ is a monotonically increasing function of $m$ because of renormalization. So we define the $p$-energy function as
$$\E_p(u) = \lim\limits_{m \to \infty} \E_p^{(m)}(u), $$ which exists for all $u$ as an extended real number. Now we define $\dom(\E_p)$ as the space of continuous functions $u$ satisfying $\E_p(u) < \infty.$ In \cite{hps}, it is shown that $\dom(\E_p)$ modulo constant functions forms a Banach space endowed with the norm $\|\cdot\|_{\E_p}$ defined as $$\|u\|_{\E_p} = \E_p(u)^{1/p}.$$ Now we proceed to define the energy form from the energy function as
\begin{equation}\label{eq-6}
  \E_p(u,v) \coloneqq \frac{1}{p}~\left.\frac{\mathrm d}{\mathrm d t} \E_p(u+tv)\right|_{t=0}.
\end{equation}
Note that we do not know whether $\E_p(u+tv)$ is differentiable or not but we know by the convexity of $A_p$ that $\E_p(u)$ is a convex function. So, we interpret the equation \eqref{eq-6} as an interval-valued equation. That is, $$\E_p(u,v) = [\E^-_p(u,v), \E^+_p(u,v)]$$ is a nonempty compact interval and the end points are the one sided derivatives. Also, it satisfies the following properties
\begin{enumerate}[(i)]
  \item $\E_p(u,av) = a~\E_p(u,v)$
  \item $\E_p(u,v_1 + v_2) \subseteq \E_p(u,v_1) + \E_p(u,v_2)$
  \item $\E_p(u,u) = \E_p(u)$
\end{enumerate}

We recall some results which will be required to prove our results.
\begin{lemma}\cite[Lemma 3.2]{sw}\label{lem-2}
There exists a constant $K_p$ such that for all $u \in \dom(\E_p)$ we have
\begin{equation*}
|u(x) - u(y)| \leq K_p \E_p(u)^{1/p}(r_p^{1/p})^m
\end{equation*}
whenever $x$ and $y$ belong to the same or adjacent cells of order $m$.
\end{lemma}
Let $\dom_0(\E_p)$ be the subspace of $\dom(\E_p)$ containing all functions which vanish at the boundary.
\begin{lemma}\label{lem-5}
If $u \in \dom_0(\E_p)$, then there exists a real positive constant $K$ such that $\|u\|_\infty \leq K \|u\|_{\E_p}.$
\end{lemma}
\begin{proof}
We can connect a point on $\cup_{j \geq 0}F^{j}(\s_0)$ and a boundary point by a string of points. As boundary values are zero, using triangle inequality, Lemma \ref{lem-2} and the fact that $0<r_p<1$, we get the result.
\end{proof}
Now we define a weak solution for the problem \eqref{prob}.
\begin{definition}
   We say $u \in\dom_0(\E_p)$ is a weak solution to the problem \eqref{prob} if it satisfies $$\lambda \int\limits_{\s} f(x)|u|^{q-2}uv \dm + \int\limits_{\s} g(x)|u|^{l-2}u v \dm \in M(\|u\|_{\E_p}^p)\E_p(u,v)$$ for all
  $v \in \dom_0(\E_p).$
\end{definition}
Our main results include:
\begin{theorem}\label{main}
There exists a $\la_1 >0$ such that problem \eqref{prob} has at least two nontrivial weak solutions whenever $0 < \la <\la_1.$
\end{theorem}
\section{The Euler functional and its analysis}
Let $\la >0$, $p > 1$ and $1 <q<p<l.$ The Euler functional associated with the problem \eqref{prob} for $u \in \dom_0(\E_p)$ is defined as
$$ I_{\la,M}(u) = \frac{1}{p}\hat{M}(\|u\|_{\E_p}^p) - \frac{\la}{q} \int\limits_{\s} f(x)|u|^{q} \dm - \frac{1}{l} \int\limits_{\s} g(x)|u|^{l} \dm.$$
where $\hat{M}(s) = \int_{0}^{s}(a t^k +b)dt = \frac{1}{k+1} a s^{k+1} + b s.$
So, we write $I_{\la,M}$ in explicit term as
\begin{equation}\label{ILA}
I_{\la,M}(u) = \frac{1}{p}\left(\frac{a }{k+1}\|u\|_{\E_p}^{p(k+1)} + b \|u\|_{\E_p}^p\right) - \frac{\la}{q} \int\limits_{\s} f(x)|u|^{q} \dm - \frac{1}{l} \int\limits_{\s} g(x)|u|^{l} \dm
\end{equation}
We do not know the range of $I_{\la,M}$ on $\dom_0(\E_p).$ So, we will consider a set where it is bounded below and do our analysis. Consider the set
\begin{align*}
 \M_{\la,M}(\s) &= \left\{ u \in \dom_0(\E_p)\setminus\{0\} : \lambda \int\limits_{\s} f(x)|u|^{q} \dm + \int\limits_{\s} g(x)|u|^{l} \dm \in M(\|u\|_{\E_p}^p)\E_p(u,u)\right\}   \\
        &= \left\{u \in \dom_0(\E_p)\setminus\{0\} : \lambda \int\limits_{\s} f(x)|u|^{q} \dm + \int\limits_{\s} g(x)|u|^{l} \dm = M(\|u\|_{\E_p}^p)\|u\|_{\E_p}^p\right\}.
\end{align*}
This means $u \in \M_{\la,M}(\s)$ if and only if
\begin{equation}\label{eq-1}
  M(\|u\|_{\E_p}^p)\|u\|_{\E_p}^p - \lambda \int\limits_{\s} f(x)|u|^{q} \dm - \int\limits_{\s} g(x)|u|^{l} \dm = 0.
\end{equation}
Using equation \eqref{eq-1}, we get the following as a consequence on $\M_{\la,M}(\s)$
\begin{equation}\label{eq-2}
\begin{split}
I_{\la,M}(u) &= \frac{1}{p}\left(\frac{a }{k+1}\|u\|_{\E_p}^{p(k+1)} + b \|u\|_{\E_p}^p\right) - \frac{\la}{q} \int\limits_{\s} f(x)|u|^{q} \dm - \frac{1}{l} \left(M(\|u\|_{\E_p}^p)\|u\|_{\E_p}^p - \lambda \int\limits_{\s} f(x)|u|^{q} \dm\right)\\
            &=\left(\frac{1}{p(k+1)}-\frac{1}{l}\right)a\|u\|_{\E_p}^{p(k+1)} + \left(\frac{1}{p}-\frac{1}{l}\right) b\|u\|_{\E_p}^p - \left(\frac{1}{q}-\frac{1}{l}\right) \la \int\limits_{\s} f(x)|u|^{q} \dm.
\end{split}
\end{equation}
At the same time we get the following equation as well
\begin{equation}\label{eq-3}
\begin{split}
I_{\la,M}(u) &= \frac{1}{p}\left(\frac{a }{k+1}\|u\|_{\E_p}^{p(k+1)} + b \|u\|_{\E_p}^p\right) - \frac{1}{q} \left(M(\|u\|_{\E_p}^p)\|u\|_{\E_p}^p- \int\limits_{\s} g(x)|u|^{l} \dm\right) - \frac{1}{l} \int\limits_{\s} g(x)|u|^{l} \dm\\
             &=\left(\frac{1}{p(k+1)}-\frac{1}{q}\right)a\|u\|_{\E_p}^{p(k+1)} + \left(\frac{1}{p}-\frac{1}{q}\right) b\|u\|_{\E_p}^p + \left(\frac{1}{q}-\frac{1}{l}\right) \int\limits_{\s} g(x)|u|^{l} \dm.
\end{split}
\end{equation}
\begin{definition}
Let $\B$ be a Banach space. A functional $f : \B \to \R$ is said to be coercive if $f(x) \to \infty$ as $\|x\| \to \infty.$
\end{definition}
\begin{theorem}
$I_{\la,M}$ is coercive and bounded from below on $\M_{\la,M}(\s).$
\end{theorem}
\begin{proof}
  From \eqref{eq-2}, $0<q<p,k+1<p(k+1)<l$, continuity of $u$, boundedness of $f$ and Lemma \ref{lem-5} we get
  \begin{align*}
 I_{\la,M}(u) &=\left(\frac{1}{p(k+1)}-\frac{1}{l}\right)a\|u\|_{\E_p}^{p(k+1)} + \left(\frac{1}{p}-\frac{1}{l}\right) b\|u\|_{\E_p}^p - \left(\frac{1}{q}-\frac{1}{l}\right) \la \int\limits_{\s} f(x)|u|^{q} \dm\\
         &\geq a\|u\|_{\E_p}^{p(k+1)}\left(\frac{1}{p(k+1)} - \frac{1}{l}\right) + b \|u\|_{\E_p}^p \left(\frac{1}{p} -\frac{1}{l}\right) - \left(\frac{\la}{q} - \frac{\la}{l}\right)\|u\|_{\E_p}^{q} K^q \|f\|_1\mu(\s)
\end{align*}
Hence, we conclude that $I_{\la,M}$ is coercive and bounded from below on $\M_{\la,M}$.
\end{proof}
For any $u \in \dom_0(\E_p)$, now consider the map $\phi_u : (0,\infty) \to \R$ defined as $\phi_u(t) = I_{\la,M}(tu),$ that is,
$$\phi_u(t) = \frac{1}{p}\left(\frac{at^{p(k+1)}}{k+1} \|u\|_{\E_p}^{p(k+1)} + b t^p \|u\|_{\E_p}^p\right) - \frac{\la t^q}{q} \int\limits_{\s} f(x)|u|^{q} \dm - \frac{t^l}{l} \int\limits_{\s} g(x)|u|^{l} \dm.$$

As $\phi_u$ is a smooth function we can compute its derivatives at $t=1$ as below.
\begin{equation*}
\begin{split}
  \phi_u'(1) &= a \|u\|_{\E_p}^{p(k+1)} + b \|u\|_{\E_p}^p - \la \int\limits_{\s} f(x)|u|^{q} \dm - \int\limits_{\s} g(x)|u|^{l} \dm~~~~~~ \text{and}\\
   \phi_u''(1) &=(p(k+1)-1) a\|u\|_{\E_p}^{p(k+1)} + (p-1)b \|u\|_{\E_p}^p - \la (q-1)\int\limits_{\s} f(x)|u|^{q} \dm - (l-1)\int\limits_{\s} g(x)|u|^{l} \dm.
\end{split}
\end{equation*}
It is easy to observe that $u \in \M_{\la,M}(\s)$ if and only if $\phi_u'(1) = 0$ and more generally, we can say the following
\begin{lemma}\label{lem-1}
  $tu \in \M_{\la,M}(\s)$ if and only if $\phi'_{tu}(1) = 0$, that is, $\phi'_u(t) = 0$.
\end{lemma}

To make our study easier, we will subdivide $\M_{\la,M}(\s)$ into sets corresponding to local minima, local maxima and points of inflection at $t=1.$ Hence we define the following sets.
\begin{align*}
  \M_{\la,M}^+(\s) &= \{u \in \M_{\la,M}(\s) : \phi_u''(1) > 0\}, \\
  \M_{\la,M}^0(\s) &= \{u \in \M_{\la,M}(\s) : \phi_u''(1) = 0\}  \\
  \text{and } \M_{\la,M}^-(\s) &= \{u \in \M_{\la,M}(\s) : \phi_u''(1) < 0\}.
\end{align*}

As $\phi_u'(1) = 0$, we get the following
\begin{align*}
  \phi_u''(1) &= (p(k+1)-l)a\|u\|_{\E_p}^{p(k+1)} + (p-l)b\|u\|_{\E_p}^p + \la(l-q)\int\limits_{\s} f(x)|u|^{q} \dm~ \text{and} \\
  \phi_u''(1) &= (p(k+1)-q)a\|u\|_{\E_p}^{p(k+1)} + (p-q)b\|u\|_{\E_p}^p + (q-l)\int\limits_{\s} g(x)|u|^{l} \dm.
\end{align*}
\section{Analysis of the map $\phi_u$}
In this section we will study the mapping $\phi_u$ with respect to our problem. The graphs of this function is determined by certain factors that includes $0<q<p,k+1<p(k+1)<l, \textbf{f} \coloneqq\int\limits_{\s} f(x)|u|^{q} \dm$ and $\textbf{g} \coloneqq\int\limits_{\s} g(x)|u|^{l}\dm.$\\
\uline{Case I} : Let $\int_{\s}f(x)|u|^{q}\dm \leq 0$ and $\int_{\s}g(x)|u|^{l}\dm \leq 0. $ Then we observe that $\phi_u(t)$ is an increasing function of $t$. So, no multiple of $u$ can be in $\M_{\la}(\s)$ using Lemma \ref{lem-1}.\\
\uline{Case II} : Let $\int_{\s}f(x)|u|^{q}\dm > 0$ and $\int_{\s}g(x)|u|^{l}\dm \leq 0.$ We observe that $\phi_u(t)$ decreases first and then increases. Also, $\phi_u(t) \to +\infty$ as $t \to +\infty.$ Thus $\phi_u$ has only one positive root. Thus the fibering map $\phi_u$ has a unique critical point at $t_1$ which is a local minimum. Hence, $t_1 u  \in \M_{\la,M}(\s),$ in particular, $\M_{\la,M}^{+}(\s).$  \\
\uline{Case III} : Let $\int_{\s}f(x)|u|^{q}\dm \leq 0$ and $\int_{\s}g(x)|u|^{l}\dm > 0.$ We observe that $\phi_u(t)$ increases initially and then decreases and goes to $- \infty$ as $t \to +\infty.$ Thus $\phi_u$ has only one positive root. Thus the fibering map $\phi_u$ has a unique critical point at $t_2$ which is a local maximum. Then $t_2 u  \in \M_{\la,M}(\s),$ in particular, $\M_{\la,M}^{-}(\s).$\\
\uline{Case IV} : Let $\int_{\s}f(x)|u|^{q}\dm > 0$ and  $\int_{\s}g(x)|u|^{l}\dm > 0.$ We observe that $\phi_u(t)$ decreases initially, then increases and then decreases after wards. Also, $\phi_u(t)\to -\infty$ as $t \to +\infty.$ For small enough value of $\la$ it has exactly two positive real root. Thus the fibering map $\phi_u$ has exactly two critical points at, let say, $t_3$  and $t_4.$ which is a local minimum and maximum respectively. Then $t_3 u, t_4u  \in \M_{\la,M}(\s),$ in particular, $t_3 u \in \M_{\la,M}^{+}(\s)$ and $t_4u \in \M_{\la,M}^{-}(\s).$\\
The following results gives us a range of $\la$ such that the set $\M_{\la,M}^0(\s)$ is an empty set.
\begin{lemma}\label{lem-6}
If $u \in \M_{\la,M}^{0},$ then $\la\int_{\s} f(x)|u|^q \dm >0$ and $\int_{\s} g(x)|u|^l \dm >0.$
\end{lemma}
\begin{proof}
$u \in \M_{\la,M}^{0},$ then $\phi_{u}''(1)=0$
\begin{align*}
&\implies (p(k+1)-1) a\|u\|_{\E_p}^{p(k+1)} + (p-1)b \|u\|_{\E_p}^p - \la (q-1)\int\limits_{\s} f(x)|u|^{q} \dm - (l-1)\int\limits_{\s} g(x)|u|^{l} \dm =0\\
&\implies (p(k+1)-l) a\|u\|_{\E_p}^{p(k+1)} + (p-l)b \|u\|_{\E_p}^p - \la (q-l)\int\limits_{\s} f(x)|u|^{q} \dm =0\\
&\implies \la \int\limits_{\s} f(x)|u|^{q} \dm = \frac{1}{l-q}\left((l-p(k+1)) a\|u\|_{\E_p}^{p(k+1)} + (l-p)b \|u\|_{\E_p}^p \right) >0.
\end{align*}
Similarly,
\begin{align*}
&\implies (p(k+1)-1) a\|u\|_{\E_p}^{p(k+1)} + (p-1)b \|u\|_{\E_p}^p - \la (q-1)\int\limits_{\s} f(x)|u|^{q} \dm - (l-1)\int\limits_{\s} g(x)|u|^{l} \dm =0\\
&\implies (p(k+1)-1) a\|u\|_{\E_p}^{p(k+1)} + (p-1)b \|u\|_{\E_p}^p - (q-1)\left(a\|u\|_{\E_p}^{p(k+1)} + b \|u\|_{\E_p}^p-\int\limits_{\s} g(x)|u|^{l} \dm \right)\\
&\quad\quad - (l-1)\int\limits_{\s} g(x)|u|^{l} \dm =0\\
&\implies (p(k+1)-q) a\|u\|_{\E_p}^{p(k+1)} + (p-q)b \|u\|_{\E_p}^p - (l-q)\int\limits_{\s} g(x)|u|^{l} \dm  =0\\
&\implies \int\limits_{\s} g(x)|u|^{l} \dm =\frac{1}{l-q}\left((p(k+1)-q) a\|u\|_{\E_p}^{p(k+1)} + (p-q)b \|u\|_{\E_p}^p\right)>0.
\end{align*}
\end{proof}

\begin{lemma}\label{lem-4}
There exists a real number $\la_1 > 0$ such that if $0 < \la < \la_1,$ then $\M_{\la,M}^0(\s)$ is an empty set.
\end{lemma}
\begin{proof}
If $\M_{\la,M}^0(\s)$ is a nonempty set, then let $u \in \M_{\la,M}^0(\s).$ From the above lemma we get,
\begin{align*}
0< \frac{1}{l-q}\left((l-p(k+1)) a\|u\|_{\E_p}^{p(k+1)} + (l-p)b \|u\|_{\E_p}^p \right) =\la \int\limits_{\s} f(x)|u|^{q} \dm \leq \la \|f\|_1K^q\|u\|_{\E_p}^q.
\end{align*}
This implies
\begin{equation}\label{eq_com1}
  \|u\|_{\E_p} \leq \left(\frac{\la(l-q)\|f\|_1 K^q}{(l-pk-p)a}\right)^{\rfrac{1}{p(k+1)-q}}\text{ and } \|u\|_{\E_p} \leq \left(\frac{\la(l-q)\|f\|_1 K^q}{(l-p)b}\right)^{\rfrac{1}{p-q}}.
\end{equation}
Also we get, 
$$\frac{1}{l-q}\left((p(k+1)-q) a\|u\|_{\E_p}^{p(k+1)} + (p-q)b \|u\|_{\E_p}^p\right) = \int\limits_{\s} g(x)|u|^{l} \dm.$$
This implies 
\begin{equation}\label{eq_com2}
  \|u\|_{\E_p} \geq \left(\frac{(pk+p-q)a}{\|g\|_1K^l(l-q)}\right)^{\rfrac{1}{l-p(k+1)}} \text{ and } \|u\|_{\E_p} \geq \left(\frac{(p-q)b}{\|g\|_1K^l(l-q)}\right)^{\rfrac{1}{l-p}}.
\end{equation}
Comparing first part of Eq. \eqref{eq_com1} to first part of Eq. \eqref{eq_com2} we get, 
$$\la \geq \left(\frac{(pk+p-q)a}{\|g\|_1K^l(l-q)}\right)^{\frac{pk+p-q}{l-pk-p}} \frac{(l-pk-p)a}{(l-q)\|f\|_1K^q} \coloneqq \la_2.$$
Comparing second part of Eq. \eqref{eq_com1} to second part of Eq. \eqref{eq_com2} we get,
$$\la \geq \left(\frac{(p-q)b}{\|g\|_1K^l(l-q)}\right)^{\rfrac{p-q}{l-p}} \frac{(l-p)b}{(l-q)\|f\|_1K^q} \coloneqq \la_3.$$
Now if we define $\la_1 = \min\{\la_2,\la_3\}$, then we have $\la \ge \la_1$. Therefore, for $0 < \la < \la_1$ $\M_{\la,M}^0(\s)$ must be an empty set.
\end{proof}
\begin{lemma}\label{lem-7}
There exists a positive real number $\widehat{\la}_1$ such that $\inf_{u\in \M_{\la,M}^-}I_{\la,M}(u) > \de_1$ for all $0<\la<\widehat{\la}_1$. 
\end{lemma}
\begin{proof}
Let $u \in \M_{\la,M}^-.$ Putting the second inequality of \eqref{eq_com2} in Eq.\eqref{eq-2}, we get
\begin{align*}
I_{\la,M}(u) &=\left(\frac{1}{p(k+1)}-\frac{1}{l}\right)a\|u\|_{\E_p}^{p(k+1)} + \left(\frac{1}{p}-\frac{1}{l}\right) b\|u\|_{\E_p}^p - \left(\frac{1}{q}-\frac{1}{l}\right) \la \int\limits_{\s} f(x)|u|^{q} \dm\\
		&\geq \left(\frac{1}{p}-\frac{1}{l}\right) b\|u\|_{\E_p}^p - \left(\frac{1}{q}-\frac{1}{l}\right) \la \|f\|_1K^q\|u\|_{\E_p}^{q}\\
		&= \|u\|_{\E_p}^q \left(\left(\frac{1}{p}-\frac{1}{l}\right) b\|u\|_{\E_p}^{p-q} - \left(\frac{1}{q}-\frac{1}{l}\right) \la \|f\|_1K^q\right)\\
		&\geq \left(\frac{(p-q)b}{(l-q)K^l\|g\|_1}\right)^\frac{q}{l-p} \left(\left(\frac{1}{p}-\frac{1}{l}\right) b\left(\frac{(p-q)b}{(l-q)K^l\|g\|_1}\right)^\frac{p-q}{l-p} - \left(\frac{1}{q}-\frac{1}{l}\right) \la \|f\|_1K^q\right)\\
		&\coloneqq \de_1>0.
\end{align*}
If \begin{align*}
& 0< \left(\frac{1}{p}-\frac{1}{l}\right) b\left(\frac{(p-q)b}{(l-q)K^l\|g\|_1}\right)^\frac{p-q}{l-p} - \left(\frac{1}{q}-\frac{1}{l}\right) \la \|f\|_1K^q \\
i.e., & \left(\frac{1}{q}-\frac{1}{l}\right) \la \|f\|_1K^q < \left(\frac{1}{p}-\frac{1}{l}\right) b\left(\frac{(p-q)b}{(l-q)K^l\|g\|_1}\right)^\frac{p-q}{l-p}  \\
i.e., & ~\la < \left(\frac{q}{p}\right)\frac{(l-p)b}{(l-q)K^q \|f\|_1} \left(\frac{(p-q)b}{(l-q)K^l\|g\|_1}\right)^\frac{p-q}{l-p} = \left(\frac{q}{p}\right) \la_3.
\end{align*}
Similarly, putting the first inequality of \eqref{eq_com2} in Eq. \eqref{eq-2}, we get
\begin{align*}
I_{\la,M}(u) &=\left(\frac{1}{p(k+1)}-\frac{1}{l}\right)a\|u\|_{\E_p}^{p(k+1)} + \left(\frac{1}{p}-\frac{1}{l}\right) b\|u\|_{\E_p}^p - \left(\frac{1}{q}-\frac{1}{l}\right) \la \int\limits_{\s} f(x)|u|^{q} \dm\\
	&\geq \left(\frac{1}{p(k+1)}-\frac{1}{l}\right)a\|u\|_{\E_p}^{p(k+1)} - \left(\frac{1}{q}-\frac{1}{l}\right) \la \|f\|_1K^q\|u\|_{\E_p}^{q} \\
&\geq \left(\frac{(p(k+1)-q)a}{\|g\|_1K^l(l-q)}\right)^{\frac{q}{l-p(k+1)}}\left(\left(\frac{1}{p(k+1)}-\frac{1}{l}\right)a\left(\frac{(p(k+1)-q)a}{\|g\|_1K^l(l-q)}\right)^{\frac{p(k+1)-q}{l-p(k+1)}} - \left(\frac{1}{q}-\frac{1}{l}\right) \la \|f\|_1K^q \right)
\coloneqq \de_1>0.
\end{align*}
If
\begin{align*}
& 0<\left(\frac{1}{p(k+1)}-\frac{1}{l}\right)a\left(\frac{(p(k+1)-q)a}{\|g\|_1K^l(l-q)}\right)^{\frac{p(k+1)-q}{l-p(k+1)}} - \left(\frac{1}{q}-\frac{1}{l}\right) \la \|f\|_1K^q\\
i.e., & \left(\frac{1}{q}-\frac{1}{l}\right) \la \|f\|_1K^q < \left(\frac{1}{p(k+1)}-\frac{1}{l}\right)a\left(\frac{(p(k+1)-q)a}{\|g\|_1K^l(l-q)}\right)^{\frac{p(k+1)-q}{l-p(k+1)}} \\
i.e., & ~\la < \left(\frac{q}{p(k+1)}\right)\frac{(l-p(k+1))a}{(l-q)\|f\|_1K^q } \left(\frac{(p(k+1)-q)a}{\|g\|_1K^l(l-q)}\right)^{\frac{p(k+1)-q}{l-p(k+1)}} = \left(\frac{q}{p(k+1)}\right) \la_2.\\
\end{align*}
Hence define $\widehat{\la}_1 = \min\left\{\left(\frac{q}{p}\right) \la_3, \left(\frac{q}{p(k+1)}\right) \la_2\right\}.$ This completes the proof.
\end{proof}
\section{Proof of the Main Results}
\begin{theorem}\label{thm-1}
If $0<\la < \la_1,$ then there exists a minimizer of $I_{\la,M}$ on $\M_{\la,\M}^{+}(\s).$
\end{theorem}
\begin{proof}

As we have shown that $I_{\la,M}$ is bounded below on $\M_{\la,M}(\s)$, it is bounded below on $\M_{\la,M}^{+}(\s)$ as well. Thus there exists a sequence $\{u_n\} \subset \M_{\la,M}^{+}(\s)$ such that
$$\lim\limits_{n \to \infty}I_{\la,M}(u_n) = \inf\limits_{u \in \M_{\la,M}^{+}(\s)} I_{\la,M}(u).$$
Now we claim that $\{u_n\}$ is bounded in $\dom_0(\E_p)$. If $\{u_n\}$ is unbounded then there exists a subsequence $\{u_{n_k}\}$ such that $\|u_{n_k}\|_{\E_p} \to \infty$ as $k \to \infty$ and as $I_{\la,M}$ is coercive, $I_{\la,M}(u_{n_k}) \to \infty$ as $k \to \infty.$  So, $\lim\limits_{n \to \infty}I_{\la,M}(u_n)= \lim\limits_{k \to \infty}I_{\la,M}(u_{n_k})=\infty$  which is a contradiction as $\M_{\la,M}^+(\s)$ is nonempty.\\

\uline{Claim} : Sequence of functions $\{u_n\}$ is equicontinuous.\\
By Lemma \ref{lem-2}, $|u_n(x) - u_n(y)| \leq K_p (\E_p(u))^{1/p} (r_p^{1/p})^m$ whenever $x$ and $y$ belongs to the same or adjacent cells of order $m.$ Let $ B = \sup\{\|u_n\|_{\E_p} : n \in \N\}.$ Let $\e > 0$ be given. As $0<r_p<1$ there exists $m \in \N$ such that $K_p B(r_p^m)^{1/p} < \e.$ Choose $\de = 2^{-m}$. Then $\|x-y\| < \de$ implies that $|u_n(x) - u_n(y)| < \e $ for all $n \in \N.$ Hence $\{u_n\}$ is an equicontinuous family of functions.
As $\{u_n\}\subset \dom_0(\E_p)$, uniform boundedness follows from Lemma \ref{lem-5}. By the Arzela-Ascoli theorem, there exists a subsequence of $\{u_n\}$ which we still call $\{u_n\}$ converging to a continuous function $u_0$, that is,
$$\|u_n - u_0\|_\infty \to 0~~ \text{as}~~ n \to \infty.$$
Next we claim that $u_0 \in \dom_0(\E_p).$
$$\E_p(u_0) = \sup\limits_{m} \E_p^{(m)}(u_0) =\sup\limits_{m} \lim\limits_{n \to \infty} \E_p^{(m)}(u_n) \leq \sup\limits_{m} \limsup\limits_{n \to \infty} \E_p(u_n) = \limsup\limits_{n \to \infty} \E_p(u_n) $$
As $\limsup\limits_{n \to \infty} \E_p(u_n) < +\infty$, we get the claim.\\
If we choose $u \in \dom_0(\E_p)$ such that $\int_{\s}f(x)|u|^{q} \dm >0,$ then it falls under case II or IV. So there exists $t_u$ such that $t_u u \in \M_{\la,M}^+(\s)$ and $I_{\la,M}(t_u u) <0.$ Hence, $\inf\limits_{u \in \M_{\la,M}^+(\s)} I_{\la,M}(u) < 0.$
By equation \eqref{eq-2}
$$I_{\la,M}(u_n) =\left(\frac{1}{p(k+1)}-\frac{1}{l}\right)a\|u_n\|_{\E_p}^{p(k+1)} + \left(\frac{1}{p}-\frac{1}{l}\right) b\|u_n\|_{\E_p}^p - \left(\frac{1}{q}-\frac{1}{l}\right) \la \int\limits_{\s} f(x)|u_n|^{q} \dm $$
and this implies
$$\left(\frac{1}{q}- \frac{1}{l}\right) \la \int\limits_{\s} f(x)|u_n|^{q} \dm =\left(\frac{1}{p(k+1)}-\frac{1}{l}\right)a\|u_n\|_{\E_p}^{p(k+1)} + \left(\frac{1}{p}-\frac{1}{l}\right) b\|u_n\|_{\E_p}^p - I_{\la,M}(u_n) > -I_{\la,M}(u_n)$$
Taking limit as $n \to \infty$, we see that $\int_{\s}f(x)|u_0|^{q}\dm >0.$ Hence, this falls in case II  or III. Hence there exists $t_0 > 0$ such that $t_0u_0 \in \M_{\la,M}^+(\s)$ and $\phi'_{u_0}(t_0) = 0.$
By the Lebesgue dominated convergence theorem we have $$\lim\limits_{n \to \infty} \int_{\s} f(x)|u_n|^{q} \dm = \int_{\s} f(x) |u_0|^{q} \dm$$ and
$$\lim\limits_{n \to \infty} \int_{\s} g(x)|u_n|^{l} \dm = \int_{\s} g(x) |u_0|^{l} \dm.$$
If we assume $\E_p(u_0) < \limsup\limits_{n \to \infty} \E_p(u_n)$ then we get $$\phi'_{u_0}(t) < \limsup\limits_{n \to \infty} \phi'_{u_n}(t).$$
Since $\{u_n\} \subset \M_{\la,M}^{+}(\s),~\phi'_{u_n}(1)= 0$ for all $n \in \N.$
It follows from the above assumption that $\limsup\limits_{n \to \infty} \phi '_{u_n}(t_0) > \phi '_{u_0}(t_0) = 0$ which implies that $\phi '_{u_n}(t_0) > 0$ for some $n.$ Hence $t_0 \geq 1.$
As $t_0 u_0 \in \M_{\la,M}^+(\s),$ we get the following
\begin{align*}
  \inf\limits_{u \in \M_{\la,M}^+(\s)} I_{\la,M}(u) &\leq I_{\la,M}(t_0 u_0) = \phi_{u_0}(t_0) < \phi_{u_0}(1) < \limsup\limits_{n \to \infty} \phi_{u_n}(1)\\
   &= \limsup\limits_{n \to \infty} I_\la(u_n) = \lim\limits_{n \to \infty} I_\la(u_n) = \inf\limits_{u \in \M_{\la,M}^+(\s)} I_{\la,M}(u)
\end{align*}
which is a contradiction. Thus $t_0 = 1$ and $\E_p(u_0) = \limsup\limits_{n \to \infty} \E_p(u_n).$\\
So, $$I_{\la,M}(u_0) = \limsup\limits_{n \to \infty} I_{\la,M}(u_n) = \lim\limits_{n \to \infty} I_{\la,M}(u_n) = \inf\limits_{u \in \M_{\la,M}^+(\s)} I_{\la,M}(u).$$
Hence, $u_0$ is a minimizer of $I_{\la,M}$ on $\M_{\la,M}^+(\s).$
\end{proof}
\begin{theorem}\label{thm-2}
  If $0 < \la < \widehat{\la}_1$, then there exists a minimizer of $I_{\la,M}$ on $\M_{\la,M}^{-}(\s).$
\end{theorem}

\begin{proof}
By Lemma \ref{lem-7}, we have $I_{\la,M}(v) \geq \de_1 > 0$ for all $v \in \M_{\la,M}^{-} (\s).$ So, $\inf\limits_{v \in \M_{\la,M}^{-} (\s)} I_{\la,M}(v) \geq \de_1$ and there exists a sequence $\{v_n\} \subset \M_{\la,M}^-(\s)$ such that
$$ \lim\limits_{n \to \infty} I_{\la,M}(v_n) = \inf\limits_{v \in \M_{\la,M}^{-} (\s)} I_{\la,M}(v).$$
Using similar arguments as in the last theorem, we can show that $\{v_n\}$ is bounded in $\dom_0(\E_p),~ \{v_n\}$ converges to $v_0$ uniformly and $v_0 \in \dom_0(\E_p)$. Hence $v_0$ is continuous.\\
By \eqref{eq-3}
$$I_{\la,M}(u) =\left(\frac{1}{p(k+1)}-\frac{1}{q}\right)a\|u\|_{\E_p}^{p(k+1)} + \left(\frac{1}{p}-\frac{1}{q}\right) b\|u\|_{\E_p}^p + \left(\frac{1}{q}-\frac{1}{l}\right) \int\limits_{\s} g(x)|u|^{l} \dm.$$
This implies,
$$ \left(\frac{1}{q}-\frac{1}{l}\right) \int\limits_{\s} g(x)|u|^{l} \dm =I_{\la,M}(u)- \left(\frac{1}{p(k+1)}-\frac{1}{q}\right)a\|u\|_{\E_p}^{p(k+1)} - \left(\frac{1}{p}-\frac{1}{q}\right) b\|u\|_{\E_p}^p \geq I_{\la,M}(u)\geq \de_1 > 0.$$
Taking limit as $n \to \infty$, we see that $\int_{\s}g(x)|v_0|^{l}\dm >0.$ Since $\int_{\s}g(x)|v_0|^{l}\dm >0,$ then it falls in case III or IV. Hence there exists a $t_1 > 0$ such that $t_1 v_0 \in \M_{\la,M}^-(\s)$ and $\phi'_{v_0}(t_1) = 0.$
Using the Lebesgue dominated convergence theorem we have $$\lim\limits_{n \to \infty} \int_{\s} f(x)|v_n|^{q} \dm = \int_{\s} f(x) |v_0|^{q} \dm$$ and
$$\lim\limits_{n \to \infty} \int_{\s} g(x)|v_n|^{l} \dm = \int_{\s} g(x) |v_0|^{l} \dm.$$
If we assume $\E_p(v_0) < \limsup\limits_{n \to \infty} \E_p(v_n)$, then we get $\phi'_{v_0}(t) < \limsup\limits_{n \to \infty} \phi'_{v_n}(t).$
Since $\{v_n\} \subset \M_{\la,M}^{-}(\s)$, $\phi'_{v_n}(1)= 0$ for all $n \in \N.$
It follows from the above assumption that $\limsup\limits_{n \to \infty} \phi '_{v_n}(t_1) > \phi '_{v_0}(t_1) = 0$ which implies that $\phi '_{v_n}(t_1) > 0$ for some $n.$ Hence $t_1 \leq 1.$ As $t_1 v_0 \in \M_{\la,M}^-(\s)$ we get the following
\begin{align*}
  I_{\la,M}(t_1 v_0) &= \phi_{v_0}(t_1) < \limsup\limits_{n \to \infty} \phi_{v_n}(t_1) \leq \limsup\limits_{n\to \infty} \phi_{v_n}(1)\\
   &= \limsup\limits_{n \to \infty} I_{\la,M}(v_n) = \lim\limits_{n \to \infty} I_{\la,M}(v_n) = \inf\limits_{v \in \M_{\la,M}^{-}(\s)} I_{\la,M}(v)
\end{align*}
which is a contradiction. Thus $\E_p(v_0) = \limsup\limits_{n \to \infty} \E_p(v_n).$\\
So $\phi'_{v_0}(1) =0$ and $\phi''_{v_0}(1) \leq 0.$ As we know from Lemma \ref{lem-4}, $\M_{\la,M}^0(\s) = \emptyset$ for all $\la < \la_1$ then $\phi''_{v_0}(1) <0.$
So, $v_0 \in \M_{\la,M}^-(\s)$ and $$I_{\la,M}(v_0) = \limsup\limits_{n \to \infty} I_{\la,M}(v_n) = \lim\limits_{n \to \infty} I_{\la,M}(v_n) = \inf\limits_{v \in \M_{\la,M}^{-}(\s)} I_{\la,M}(v)$$
Hence, $v_0$ is a minimizer for $I_{\la,M}$ on $\M_{\la,M}^{-}(\s).$
\end{proof}
\begin{lemma}\label{lem-3}
For each $w \in \dom_0(\E_p),$ the following hold true :
\begin{enumerate}[(i)]
\item there exists $\e_0 > 0$ such that for each $\e \in (-\e_0, \e_0)$ there exists a unique $\bar{t_\e}>0$ such that $\bar{t_\e}(u_0 + \e w) \in \M_{\la,M}^{+}(\s),$ where $u_0 \in \M_{\la,M}^{+}(\s)$ and $I_{\la,M}(u_0) = \inf\limits_{u \in \M_{\la,M}^{+}(\s)}I_{\la,M}(u).$ Also, $\bar{t_\e} \to 1$ as $\e \to 0.$
\item there exists $\e_1 > 0$ such that for each $\e \in (-\e_1, \e_1)$ there exists a unique $t_\e >0$ such that $t_\e(v_0 + \e w) \in \M_{\la,M}^-(\s),$ where $v_0 \in \M_{\la,M}^{-}(\s)$ and $I_{\la,M}(v_0) = \inf\limits_{v \in\M_{\la,M}^{-}(\s)}I_{\la,M}(v).$ Also, $t_\e \to 1$ as $\e \to 0$.
\end{enumerate}
\end{lemma}
\begin{proof}
(i) Let us define a function
$\hat{\f} : \R^4 \times (0,\infty)   \to \R $ by
$$\hat{\f}(a_1,a_2,a_3,a_4,t) = a_1at^{p(k+1)-1} + a_2 b t^{p-1}- \la a_3 t^{q-1} - a_4 t^{l-1}.$$ Then
$$\hat{\f}(a_1,a_2,a_3,a_4,t) = a_1(p(k+1)-1)at^{p(k+1)-2} + a_2 (p-1)b t^{p-2}- \la a_3(q-1) t^{q-2} - a_4(l-1) t^{l-2}.$$
Since $u_0 \in \M_{\la,M}^+(\s),$ so $\phi'_{u_0}(1) =0$ and $\phi''_{u_0}(1) >0$. Therefore,
$$\hat{\f}\left(\|u_0\|_{\E_p}^{p(k+1)},\|u_0\|_{\E_p}^p ,\int\limits_{\s}f(x)|u_0|^{q} \dm,\int\limits_{\s} g(x)|u_0|^{l} \dm,1 \right) = \phi'_{u_0}(1) = 0$$
and
$$\frac{\partial\hat{\f}}{\partial t}\left(\|u_0\|_{\E_p}^{p(k+1)},\|u_0\|_{\E_p}^p ,\int\limits_{\s}f(x)|u_0|^{q} \dm,\int\limits_{\s} g(x)|u_0|^{l} \dm, 1 \right) = \phi''_{u_0}(1) > 0.$$
The function $\f_1(\e) = \int_{\s} f(x)|u_0 +\e w|^{q}\dm$ is a continuous function and $\f_1(0) > 0$ because $u_0 \in \M_{\la,M}^{+}(\s)$. By the continuity of $\f_1$, there exists $\e_0 >0$ such that $\f_1(\e) >0$ for all $\e \in (-\e_0, \e_0)$. So, for each $\e \in (-\e_0, \e_0)$ there exists $\bar{t_\e}$ such that $\bar{t_\e}(u_0 + \e w) \in \M_{\la,M}^+(\s).$ This implies that $$\hat{\f}\left(\|u_0 + \e w\|_{\E_p}^{p(k+1)},\|u_0 + \e w\|_{\E_p}^p ,\int\limits_{\s}f(x)|u_0 + \e w|^{q} \dm,\int\limits_{\s} g(x)|u_0 + \e w|^{l} \dm, \bar{t_\e} \right) = \phi'_{u_0 + \e w}(\bar{t_\e}) = 0.$$
By the implicit function theorem there exists an open set $A \subset (0,\infty)$ containing 1, an open set $B \subset \R^4$ containing $(\|u_0\|_{\E_p}^{p(k+1)},\|u_0\|_{\E_p}^p ,\int\limits_{\s}f(x)|u_0|^{q} \dm,\int\limits_{\s} g(x)|u_0|^{l} \dm )$ and a continuous function $g : B \to A$ such that for all $y \in B$, $\hat{\f}(y, g(y)) = 0$. So there exists a unique solution to the equation $t = g(y) \in A.$
Hence,
$$\bar{t_\e} = g\left(\|u_0 + \e w\|_{\E_p}^{p(k+1)},\|u_0 + \e w\|_{\E_p}^p ,\int\limits_{\s}f(x)|u_0 + \e w|^{q} \dm,\int\limits_{\s} g(x)|u_0 + \e w|^{l} \dm\right)$$
Letting $\e \to 0$, using the continuity of $g$, we get
$$1 = g\left(\|u_0\|_{\E_p}^{p(k+1)},\|u_0\|_{\E_p}^p ,\int\limits_{\s}f(x)|u_0|^{q} \dm,\int\limits_{\s} g(x)|u_0|^{l} \dm\right)$$
Therefore, $\bar{t_\e} \to 1$ as $\e \to 0.$\\
(ii)
This can be proved by taking the function $$\f_2(\e) = \int_{\s} g(x)|v_0 +\e w|^{l}\dm$$ in place of $\f_1(\e)$ and proceeding in a similar fashion as in the proof of (i).
\end{proof}
\begin{theorem}\label{thm-8}
If $u_0$ is a minimizer of $I_{\la,M}$ on $\M_{\la,M}^{+}(\s)$ and $v_0$ is a minimizer of $I_{\la,M}$ on $\M_{\la,M}^{-}(\s),$ then $u_0$ and $v_0$ are weak solutions to the problem \eqref{prob}.
\end{theorem}
\begin{proof}
Let $\psi \in \dom_0(\E_p).$ Using Lemma \ref{lem-3}(i), there exists $\e_0 > 0$ such that for each $\e \in (-\e_0,\e_0)$ there exists $\bar{t_\e}$ such that $I_{\la,M}(\bar{t_\e} (u_0 + \e \psi)) \geq I_{\la,M}(u_0)$ and $\bar{t_\e} \to 1$ as $\e \to 0.$ Then we have
\begin{align*}
0 &\leq \lim\limits_{\e \to 0^+} \frac{1}{\e}\left(I_{\la,M}(\bar{t_\e}(u_0 + \e \psi))- I_{\la,M}(u_0)\right) = \lim\limits_{\e \to 0^+} \frac{1}{\e}\left(I_{\la,M}(\bar{t_\e}(u_0 + \e \psi))- I_{\la,M}(\bar{t_\e}u_0) +I_{\la,M}(\bar{t_\e}u_0)-  I_{\la,M}(u_0)\right)\\
  &= \lim\limits_{\e \to 0^+} \frac{1}{\e}\left(I_{\la,M}(\bar{t_\e}(u_0 + \e \psi))- I_{\la,M}(\bar{t_\e}u_0)\right)\\
  &= \lim\limits_{\e \to 0^+} \frac{1}{\e}\left(\frac{1}{p}\left(\frac{a }{k+1}\|\bar{t_\e}(u_0 + \e \psi)\|_{\E_p}^{p(k+1)} + b \|\bar{t_\e}(u_0 + \e \psi)\|_{\E_p}^p\right) - \frac{\la}{q} \int\limits_{\s} f(x)|\bar{t_\e}(u_0 + \e \psi)|^{q} \dm - \frac{1}{l} \int\limits_{\s} g(x)|\bar{t_\e}(u_0 + \e \psi)|^{l} \dm \right) \\
  &\quad -\left(\frac{1}{p}\left(\frac{a }{k+1}\|\bar{t_\e}u_0\|_{\E_p}^{p(k+1)} + b \|\bar{t_\e}u_0\|_{\E_p}^p\right) - \frac{\la}{q} \int\limits_{\s} f(x)|\bar{t_\e}u_0|^{q} \dm - \frac{1}{l} \int\limits_{\s} g(x)|\bar{t_\e}u_0|^{l} \dm\right)\\  
  &= \lim\limits_{\e \to 0^+} \frac{1}{\e} \frac{1}{p} \left(\frac{a }{k+1}\left(\|\bar{t_\e}(u_0 + \e \psi)\|_{\E_p}^{p(k+1)}-\|\bar{t_\e}u_0\|_{\E_p}^{p(k+1)}\right) + b \left(\|\bar{t_\e}(u_0 + \e \psi)\|_{\E_p}^p - \|\bar{t_\e}u_0\|_{\E_p}^p\right)\right)  \\
  &\quad- \lim\limits_{\e \to 0^+} \frac{1}{\e}\left(\frac{\la}{q} \int\limits_{\s} f(x)|\bar{t_\e}(u_0 + \e \psi)|^{q} \dm- \frac{\la}{q} \int\limits_{\s} f(x)|\bar{t_\e}u_0|^{q} \dm \right)\\
  &\quad -\lim\limits_{\e \to 0^+} \frac{1}{\e}\left( \frac{1}{l} \int\limits_{\s} g(x)|\bar{t_\e}(u_0 + \e \psi)|^{l} \dm - \frac{1}{l} \int\limits_{\s} g(x)|\bar{t_\e}u_0|^{l} \dm\right)\\
  &= \left(a\|u_0\|_{\E_p}^{pk} + b\right)  \E_p^+(u_0,\psi) - \la \int_{\s}f(x)|u_0|^{q-2}u_0\psi \dm -  \int_{\s}g(x)|u_0|^{l-2}u_0 \psi \dm.
\end{align*}
Note that the second equality follows by using $\lim\limits_{\e \to 0^+}\frac{1}{\e} (I_{\la,M}(\bar{t_\e}u_0)-  I_{\la,M}(u_0)) = 0$ because the limit is the same as $\phi'_{u_0}(1)$, which is zero.
This implies that
$$\la \int_{\s}f(x)|u_0|^{q-2}u_0\psi \dm + \int_{\s}g(x)|u_0|^{l-2}u_0 \psi \dm\leq \left(a\|u_0\|_{\E_p}^{pk} + b\right)  \E_p^+(u_0,\psi).$$
Similarly,
\begin{align*}
    0 &\geq \lim\limits_{\e \to 0^-} \frac{1}{\e}\left(I_{\la,M}(\bar{t_\e}(u_0 + \e \psi))- I_{\la,M}(u_0)\right)
    = \lim\limits_{\e \to 0^-} \frac{1}{\e}\left(I_{\la,M}(\bar{t_\e}(u_0 + \e \psi))- I_{\la,M}(\bar{t_\e}u_0) +I_{\la,M}(\bar{t_\e}u_0)-  I_{\la,M}(u_0)\right)\\
   &= \lim\limits_{\e \to 0^-} \frac{1}{\e}\left(I_{\la,M}(\bar{t_\e}(u_0 + \e \psi))- I_{\la,M}(\bar{t_\e}u_0)\right)\\
    &= \left(a\|u_0\|_{\E_p}^{pk} + b\right)  \E_p^-(u_0,\psi) - \la \int_{\s}f(x)|u_0|^{q-2}u_0\psi \dm -  \int_{\s}g(x)|u_0|^{l-2}u_0 \psi \dm.
\end{align*}
which implies that
$$ \la \int_{\s}f(x)|u_0|^{q-2}u_0\psi \dm + \int_{\s}g(x)|u_0|^{l-2}u_0 \psi \dm\geq \left(a\|u_0\|_{\E_p}^{pk} + b\right)  \E_p^-(u_0,\psi)$$
So,
$$\left(a\|u_0\|_{\E_p}^{pk} + b\right)  \E_p^-(u_0,\psi) \leq \la \int_{\s}f(x)|u_0|^{q-2}u_0\psi \dm + \int_{\s}g(x)|u_0|^{l-2}u_0 \psi \dm\leq \left(a\|u_0\|_{\E_p}^{pk} + b\right)  \E_p^+(u_0,\psi).$$
Hence
$$\la \int_{\s}f(x)|u_0|^{q-2}u_0\psi \dm + \int_{\s}g(x)|u_0|^{l-2}u_0 \psi \dm\in \left(a\|u_0\|_{\E_p}^{pk} + b\right)  \E_p(u_0,\psi).$$ for all $\psi \in \dom_0(\E_p).$ Therefore, $u_0$ is a weak solution to the problem \eqref{prob}.
Using similar arguments as in Lemma \ref{lem-3}(ii), there exists $\e_1 > 0$ such that for each $\e \in (-\e_1,\e_1)$ there exists $t_\e$ such that $I_{\la,M}(t_\e (v_0 + \e \psi)) \geq I_{\la,M}(v_0)$ and $t_\e \to 1$ as $\e \to 0.$  Then we have
\begin{align*}
0 &\leq \lim\limits_{\e \to 0^+} \frac{1}{\e}\left(I_{\la,M}(t_\e(v_0 + \e \psi) )- I_{\la,M}(v_0)\right) = \lim\limits_{\e \to 0^+} \frac{1}{\e}\left(I_{\la,M}(t_\e(v_0 + \e \psi) )- I_{\la,M}(t_\e v_0) + I_{\la,M}(t_\e v_0) - I_{\la,M}(v_0)\right)\\
  &= \lim\limits_{\e \to 0^+} \frac{1}{\e}\left(I_{\la,M}(t_\e(v_0 + \e \psi) )- I_{\la,M}(t_\e v_0)\right) \\
  &= \lim\limits_{\e \to 0^+} \frac{1}{\e} \frac{1}{p} \left(\frac{a }{k+1}\left(\|t_\e(v_0 + \e \psi)\|_{\E_p}^{p(k+1)}-\|t_\e v_0\|_{\E_p}^{p(k+1)}\right) + b \left(\|t_\e(v_0 + \e \psi)\|_{\E_p}^p - \|t_\e v_0\|_{\E_p}^p\right)\right)  \\
  &\quad- \lim\limits_{\e \to 0^+} \frac{1}{\e}\left(\frac{\la}{q} \int\limits_{\s} f(x)|t_\e (v_0 + \e \psi)|^{q} \dm- \frac{\la}{q} \int\limits_{\s} f(x)|t_\e v_0|^{q} \dm \right)\\
  &\quad -\lim\limits_{\e \to 0^+} \frac{1}{\e}\left( \frac{1}{l} \int\limits_{\s} g(x)|t_\e (v_0 + \e \psi)|^{l} \dm - \frac{1}{l} \int\limits_{\s} g(x)|t_\e v_0|^{l} \dm\right)\\
  &= \left(a\|v_0\|_{\E_p}^{pk} + b\right)  \E_p^+(v_0,\psi) - \la \int_{\s}f(x)|v_0|^{q-2}v_0\psi \dm -  \int_{\s}g(x)|v_0|^{l-2}v_0 \psi \dm.
\end{align*}
This implies that
$$\la \int_{\s}f(x)|v_0|^{q-2}v_0\psi \dm + \int_{\s}g(x)|v_0|^{l-2}v_0 \psi \dm \leq \left(a\|v_0\|_{\E_p}^{pk} + b\right)  \E_p^+(v_0,\psi).$$
Similarly,
\begin{align*}
  0 &\geq \lim\limits_{\e \to 0^-} \frac{1}{\e}\left(I_{\la,M}(t_\e(v_0 + \e \psi) )- I_{\la,M}(v_0)\right) = \lim\limits_{\e \to 0^-} \frac{1}{\e}\left(I_{\la,M}(t_\e(v_0 + \e \psi) )- I_{\la,M}(t_\e v_0) + I_{\la,M}(t_\e v_0) - I_{\la,M}(v_0)\right)\\
  &= \lim\limits_{\e \to 0^-} \frac{1}{\e}\left(I_{\la,M}(t_\e(v_0 + \e \psi) )- I_{\la,M}(t_\e v_0)\right) \\
    &= \left(a\|v_0\|_{\E_p}^{pk} + b\right)  \E_p^-(v_0,\psi) - \la \int_{\s}f(x)|v_0|^{q-2}v_0\psi \dm -  \int_{\s}g(x)|v_0|^{l-2}v_0 \psi \dm
\end{align*}
which implies that
$$\la \int_{\s}f(x)|v_0|^{q-2}v_0\psi \dm + \int_{\s}g(x)|v_0|^{l-2}v_0 \psi \dm \geq  \left(a\|v_0\|_{\E_p}^{pk} + b\right)  \E_p^-(v_0,\psi).$$
So,
$$ \left(a\|v_0\|_{\E_p}^{pk} + b\right)  \E_p^-(v_0,\psi) \leq \la \int_{\s}f(x)|v_0|^{q-2}v_0\psi \dm + \int_{\s}g(x)|v_0|^{l-2}v_0 \psi \dm \leq \left(a\|v_0\|_{\E_p}^{pk} + b\right)  \E_p^+(v_0,\psi).$$
Hence
$$\la \int_{\s}f(x)|v_0|^{q-2}v_0\psi \dm + \int_{\s}g(x)|v_0|^{l-2}v_0 \psi \dm \in \left(a\|v_0\|_{\E_p}^{pk} + b\right)  \E_p(v_0,\psi)$$ for all $\psi \in \dom_0(\E_p).$ Therefore, $v_0$ is a weak solution to the problem \eqref{prob}.
\end{proof}

Now we give the proof of Theorem \ref{main} below.
\begin{proof}{(proof of Theorem \ref{main})}
 In Lemma \ref{lem-7} we have obtained $\widehat{\la}_1$, in Theorem \ref{thm-1} and Theorem \ref{thm-2} we have shown the existence of minimizers in the respective subsets and Theorem \ref{thm-8} we have shown that the problem \eqref{prob} has two nontrivial solutions. This completes the proof.
\end{proof}

\footnotesize
\nocite{}
\bibliographystyle{abbrv}
\bibliography{ref}
\end{document}